\documentclass[10pt]{amsart} 

\usepackage[psamsfonts]{amssymb} 
\usepackage{amsfonts,amsmath, graphicx} 
\usepackage{tensor}
\usepackage{enumerate}

\makeatletter
\newtheorem*{rep@theorem}{\rep@title}
\newcommand{\newreptheorem}[2]{%
\newenvironment{rep#1}[1]{%
 \def\rep@title{#2 \ref{##1}}%
 \begin{rep@theorem}}%
 {\end{rep@theorem}}}
\makeatother

\theoremstyle{plain}
\newtheorem{theorem}{Theorem}[section]
\newreptheorem{theorem}{Theorem}
\newtheorem{lemma}[theorem]{Lemma}
\newtheorem{proposition}[theorem]{Proposition}
\newreptheorem{proposition}{Proposition}
\newtheorem{corollary}[theorem]{Corollary}
\newreptheorem{corollary}{Corollary}

\theoremstyle{remark}

\newtheorem{example}[theorem]{Example}

\theoremstyle{definition}
\newtheorem{definition}[theorem]{Definition}

\def\Z{\mathbb Z}

\def\G{\Gamma}
\def\g{\gamma}

\def\<{\langle}
\def\>{\rangle}

\title[The lower central series of $A_\G$]{The lower central series of a right-angled Artin group}
\author{Richard D. Wade}   
\date{\today}

\begin{document}

\begin{abstract}
We give a description of Droms, Duchamp and Krob's extension of Magnus' approach to the lower central series of the free group to right-angled Artin groups. We also describe how Lalonde's extension of Lyndon words to the partially commutative setting gives a simple algorithm to find a basis for consecutive quotients of the lower central series of a RAAG.
\end{abstract}

\address{Department of Mathematics, The University of Utah, 155 S 1400 E, Salt Lake City, UT, 84112}
\email{wade@math.utah.edu}
\maketitle

\section{Introduction}

One can often translate problems concerning Lie groups to the world of Lie algebras. When we linearise a problem our life is much easier: we understand vector spaces and their endomorphisms very well, and we may use our knowledge here to give us information about the underlying Lie group. This paper looks at how such methods are also beneficial for studying discrete groups, at least in respect to their lower central series.  

 Let $G$ be any group. One may form a Lie $\mathbb{Z}$--algebra by taking the direct sum $\sum_{i=1}^\infty \g_i(G)/\g_{i+1}(G)$, where $\g_i(G)$ is the $i$th term in the lower central series, and the bracket operation is given by taking commutators in $G$. If $G$ is a free group the picture is very nice indeed. The Lie algebra one attains is a \emph{free Lie algebra}, and the structure theory of free Lie algebras allows one to obtain information about free groups and their automorphisms.

This Lie algebra correspondence is well-known. It is covered in detail in Magnus' classic textbook on combinatorial group theory \cite[Chapter 5]{MKS}, and also appears in Bourbaki \cite{MR979493}. The aim of this paper is to give a description of the analogous theory for right-angled Artin groups, or RAAGs. These can be thought of as modified free groups, where some pairs of basis elements are allowed to commute. Any graph $\G$ determines a right-angled Artin group $A_\G$ as follows: suppose that $E$ and $V$ are the edge and vertex sets of $\G$ and let $\iota$ and $\tau$ be the maps that send an edge to its initial and terminal vertices respectively. The group $A_\G$ then has the presentation: 
$$A_\G=\langle v \in V | [\iota(e),\tau(e)] : e \in E \rangle .$$
In particular, there is a generator for each vertex of $\G$ and a commutator relation corresponding to each edge. Graphs with no edges determine free groups, and complete graphs determine free abelian groups. The RAAG moniker is popular in geometric group theory but these groups also enjoy interesting combinatorial and algorithmic properties (particularly in the context of cryptography) and appear throughout the literature as  \emph{(free) partially commutative groups}, \emph{graph groups}, \emph{trace groups}, and \emph{semifree groups}.

After replacing the free Lie algebra above with a \emph{free partially commutative Lie algebra}, the description of the lower central series algebra and its resulting applications also holds in this more general setting. These results are not new, however we feel that a unified summary of key results of Droms, Duchamp, Krob, and Lalonde \cite{Droms,DK1,DK2,L1,L2,KL} will make a useful reference. It is worth noting that Papadima and Suciu also have a beautifully succinct, if less hands-on, proof of this correspondence in their work \cite{MR2207874}.

The algebraic approach in this paper has much wider implications than one might initially guess. The author uses Duchamp and Krob's work in \cite{MR3145135} to give strong restrictions on how higher-rank lattices in semisimple Lie groups can act on right-angled Artin groups. We will see below that this Lie theory viewpoint allows us to prove that RAAGs are residually torsion-free nilpotent; this is used by Linnell, Okun, and Schick in their proof of the strong Atiyah conjecture for RAAGs \cite{MR2922714}. The work of Wise and Agol shows that the fundamental group of every closed hyperbolic 3--manifold is virtually a subgroup of a RAAG \cite{MR3104553}. Such groups are therefore virtually residually torsion-free nilpotent.

The paper is set out as follows: in Section~2 we review a classical construction that builds a a Lie algebra $L_\mathcal{G}$ from any \emph{central filtration} $\mathcal{G}=\{G_i\}_{i=1}^\infty$ of a group. This is a generalisation of the construction of the Lie algebra associated to the lower central series mentioned before. It is functorial in the sense that if you have two central filtrations $\mathcal{G}=\{G_i\}$ and $\mathcal{H}=\{H_i\}$ of groups $G$ and $H$ respectively, and $\phi:G\to H$ is a homomorphism such that $\phi(G_i) \subset H_i$ for all $i$, then there is an induced algebra homomorphism $L_\mathcal{G} \to L_\mathcal{H}$. 

In Section \ref{s:cast} we build up a host of \emph{free partially commutative} objects associated to a right-angled Artin group. Of central importance is the \emph{free partially commutative monoid} $M$, which may be viewed as the monoid of positive elements in $A_\G$. We define $U$ to be the free $\mathbb{Z}$--module on $M$. The module $U$ inherits a graded algebra structure, with the grading coming from word length in $M$. One can extend $U$ to an algebra $U^\infty$ by allowing infinitely many coefficients in a sequence of elements of $M$ to be nonzero. The algebra $U^\infty$ behaves very much like an algebra of formal power series. For instance, if $v_1, \ldots, v_n$ are the generators of $A_\G$ and $\mathbf{v_1},\ldots, \mathbf{v_n}$ are the associated elements of the monoid, then $1+\mathbf{v_i}$ is a unit in $U^\infty$ with inverse $$(1+\mathbf{v_i})^{-1}=1-\mathbf{v_i}+\mathbf{v_i}^2-\mathbf{v_i}^3+\cdots.$$ If we define $U^*$ to be the group of units of $U^\infty$, the mapping $v_i \mapsto 1+\mathbf{v_i}$ gives an embedding $$\mu:A_\G \to U^*,$$ called the \emph{Magnus map}. We define a sequence of subsets $\mathcal{D}=\{D_k\}_{k=1}^{\infty}$ of $A_\G$ by saying that $g \in D_k$ if and only if $\mu(g)$ is of the form: $$\mu(g)=1+ \text{elements of $U$ of degree $\geq k$.} $$

\begin{repproposition}{p:dinfo}  For all $k$, the set $D_k$ is a subgroup of $A_\G$ and these subgroups satisfy: \begin{enumerate} \item $\mathcal{D}$ is a central filtration of $A_\G$. \item $D_{k+1} \trianglelefteq D_k$ and $D_k/D_{k+1}$ is a finitely generated free abelian group.
\item  $\gamma_k(A_\G) \subset D_k.$ \end{enumerate} \end{repproposition}

As $\mu$ is injective $\cap_{k=1}^\infty D_k=\{1\}$, and this fact combined with properties (1) and (2) imply that a right-angled Artin group is residually torsion-free nilpotent. If $\mathcal{C}$ is the central filtration given by the lower central series, then property (3) implies that we have a Lie algebra homomorphism  $L_\mathcal{C} \to L_\mathcal{D}$. We finish our study of the Magnus map by using it to give a new proof of the normal form theorem for words in right-angled Artin groups. 

The algebra $U$ has an associated Lie algebra $\mathcal{L}(U)$ consisting of the elements of $U$ and bracket operation $[a,b]=ab-ba$. In Section \ref{s:Lyndon}, we study the Lie subalgebra $L_\G$ of $\mathcal{L}(U)$ generated by the set $V=\{\mathbf{v_1},\ldots, \mathbf{v_n}\}$. For this, we use Lalonde's description of the \emph{free partially commutative Lie algebra} determined by the graph $\G$ \cite{L1,L2}.   One first defines a subset $LE(M) \subset M$ known as the set of \emph{Lyndon elements} of $M$. These have a very rigid combinatorial structure. In particular there is a way of assigning a bracketing to each Lyndon element; given a subset $X=\{x_1,\ldots,x_r\}$ of a Lie algebra $L$, this bracketing induces a $\mathbb{Z}$--module homomorphism $\phi_X:\mathbb{Z}[LE(M)] \to L$. When $X=V$, the induced map $\phi_V:\mathbb{Z}[LE(M)]\to L_\G$ is an isomorphism. This gives a basis of $L_\G$ as a free $\mathbb{Z}$--module, and allows us to give a universal defining property of $L_\G$:

\begin{reptheorem}{t:univprop}Let $L$ be a Lie algebra, and suppose that $X=\{x_1,\ldots,x_r\}$ is a subset of $L$ that satisfies $$[x_i,x_j]=0 \text{ if $v_i$ and $v_j$ are connected by an edge in $\G$.}$$ Then there is a unique algebra homomorphism $\psi_X:L_\G \to L$ such that $$\psi_X(\mathbf{v_i})=x_i \text{ for $1 \leq i \leq r$.}$$   \end{reptheorem}

We use this in Section \ref{s:end} to construct a chain of algebra homomorphisms $$L_\G \to L_\mathcal{C} \to L_\mathcal{D} \to L_\G $$
and show that the composition of the three maps is the identity on $L_\G$. In fact:

\begin{reptheorem}{t:main}
$L_\G$, $L_\mathcal{C}$, and $L_\mathcal{D}$ are isomorphic as graded Lie algebras. Furthermore, the central filtrations $\mathcal{C}$ and $\mathcal{D}$ are equal, so that $\gamma_k(A_\G)=D_k$ for all $k \geq 1$.
\end{reptheorem}

We are now able to use Lyndon elements and $L_\G$ to describe the lower central series of $A_\G$ in more detail. For instance, Proposition \ref{p:dinfo} now implies:

\begin{reptheorem}{t:tfn1}
If $k \in \mathbb{N}$, then $\g_k(A_\G)/\g_{k+1}(A_\G)$ is free-abelian, and $A_\G/\g_k(A_\G)$ is torsion-free nilpotent.
\end{reptheorem}

We have attempted to make this work as self contained as possible. In particular, we do not assume any results concerning free Lie algebras, which allows the theory of free Lie algebras and the \emph{free partially commutative Lie algebras} studied here to be developed in parallel. This comes at the cost of assuming certain facts about the combinatorics of words in RAAGs. We hope that this trade-off is beneficial for the reader.  A wonderful aspect of Magnus' approach to the study of free groups is how nicely the overall structure of his work translates to right-angled Artin groups. An avid reader is encouraged to compare Section \ref{s:Magnus map} of this paper with Section 5.5 of \cite{MKS}. The statements contained in this paper are adapted to deal with the more general setting of RAAGs, however very little work needs to be done in ensuring the proofs then follow through as well.

The author would like to thank the referee for a careful reading of the paper and helpful advice and Dawid Kielak for a series of helpful comments.

\section{Lie algebras from central filtrations} \label{s:Lie}
Let $G$ be a group.  Let $\mathcal{G}=\{G_k\}_{k\geq1}$ be a sequence of subgroups of $G$ such that for all $k,l$: \begin{align*}\text{(F1)}& \qquad G_1=G, \\  \text{(F2)}& \qquad G_{k+1}\leq G_{k}, \\ \text{(F3)}& \qquad [G_k,G_l]\subset G_{k+l}. \end{align*} We say that $\mathcal{G}$ is a \emph{central filtration}, or a \emph{central series} of $G$. The above conditions imply that $G_k \trianglelefteq G$ and $G_{k+1} \trianglelefteq G_k$ for all $k$. The results in this section are classical (see, for example, Chapter~2 of \cite{MR979493}) and we will state them without proof. 

One example of a central filtration is $\gamma(G)=\{\gamma_k(G)\}_{k\geq1}$, the lower central series of $G$. This is defined recursively by $\gamma_1(G)=G$ and $\gamma_{k+1}(G)=[G,\gamma_{k}(G)].$ Where it is clear which group we are using, we shall simply write $\gamma_k$ (or $\gamma$) rather than $\gamma_k(G)$ (or $\gamma(G)$).  An easy induction argument shows that the lower central series is contained in all central filtrations of $G$:

\begin{proposition}\label{p:centr}
Let $\mathcal{G}=\{G_k\}$ be a central filtration of $G$. Then $\g_k \subset G_k$ for all $k$.
\end{proposition}

Central filtrations tell us about residual properties of groups. We say that a central filtration $\mathcal{G}$ is separating if  $\cap_{k=1}^\infty G_k=\{1\}$. 

\begin{proposition} \label{p:cf}
Suppose that $\mathcal{G}$ is a central filtration of $G$ that is separating. Furthermore, suppose that each consecutive quotient $G_k/G_{k+1}$ is free-abelian. Then:
\begin{enumerate}[(1)]
\item $G_k$ is a normal subgroup of $G$.
\item For all $k$ the group $G/G_k$ is torsion-free nilpotent.
\item $G$ is residually torsion-free nilpotent.
\end{enumerate}
\end{proposition}

 Each central filtration $\mathcal{G}$ also gives rise to a Lie algebra $L_\mathcal{G}$ over $\mathbb{Z}$. To describe this, we first have to take a short detour to look at some commutator identities. We use the convention that for $x,y \in G$ we have $[x,y]=xyx^{-1}y^{-1}$, and for conjugation we write $^yx=yxy^{-1}.$ 

\begin{lemma}\label{l:commutator relations}
Let $x,y,z$ be elements of $G$. Then the following identities hold:
\begin{align}
\label{e:r1} \tensor*[^x]{y}{}&=[x,y].y \\
\label{e:r2} [xy,z]&= \tensor*[^x]{[y,z]}{}.[x,z]=[x,[y,z]].[y,z].[x,z], \\
\label{e:r3} [x,yz]&=[x,y].\tensor*[^y]{[x,z]}{}=[x,y].[y,[x,z]].[x,z], \end{align}

As well as the Witt--Hall identity: \begin{equation*}[[x,y],\tensor*[^y]{z}{}].[[y,z],\tensor*[^z]{x}{}].[[z,x],\tensor[^x]{y}{}]=1.\end{equation*} \end{lemma}

The reader should be aware that the above equations are different to those that occur in many group theory text books; the commutation and conjugation conventions we use are set up for left, rather than right, actions. The Witt--Hall identity implies the following `3 subgroup' theorem:

\begin{theorem}[Hall, 1933] \label{t:hall} Let $X,Y$ and $Z$ be three normal subgroups of $G$. Then $$[[X,Y],Z] \subset [[Y,Z],X].[[Z,X],Y] $$ \end{theorem}

Now let $\mathcal{G}=\{G_i\}_{i\geq1}$ be any central filtration of $G$. Let $L_{\mathcal{G},i}=G_i/G_{i+1}$. As $[G_i,G_i]\subset G_{2i}\subset G_{i+1}$ each $L_{\mathcal{G},i}$ is an abelian group, and we can form a $\mathbb{Z}$--module $L_{\mathcal{G}}=\oplus_{i=1}^\infty L_{\mathcal{G},i}.$ Any element in $L_\mathcal{G}$ is of the form $\sum_i x_i G_{i+1}$, where each $x_i \in G_i$ and only finitely many $x_i$ are not equal to the identity. 

The Witt--Hall identity can be seen as a group theoretic version of the Jacobi identity in a Lie algebra. In fact, one can use the above set of commutator relations to show the following:

\begin{proposition} The bracket operation $$[\sum_i x_i G_{i+1},\sum_j y_j G_{j+1}]=\sum_{i,j
}[x_i,y_j] G_{i+j+1}$$ gives $L_\mathcal{G}$ the structure of a graded Lie $\mathbb{Z}$--algebra. \end{proposition}

The identities \eqref{e:r2} and \eqref{e:r3} imply that if $G$ has a generating set $\{x_1,\ldots,x_n\}$ then any consecutive quotient $\g_k(G)/\g_{k+1}(G)$ of terms in the lower central series is generated by elements of the form $[x_{i_1},[x_{i_2},[\cdots[x_{i_{k-1}},x_{i_k}]\cdots]]].\g_{k+1}(G)$. In particular:

\begin{proposition} \label{p:gen}
If $G$ is generated by $\{x_1,\ldots,x_n\}$ then $L_{\g(G)}$ is generated by the set $\{x_1\g_1(G),\ldots,x_n\g_1(G)\}$.
\end{proposition}

We finish this section with a useful observation:

\begin{proposition} \label{p:func}
Let $\mathcal{G}=\{G_i\}$ and $\mathcal{H}=\{H_i\}$ be central filtrations of groups $G$ and $H$ respectively. Let $\phi:G \to H$ be a homomorphism such that $\phi(G_i)\subset \phi(H_i)$ for all $i\in \mathbb{N}$. Then $\phi$ induces a graded Lie algebra homomorphism $\Phi:L_\mathcal{G} \to L_\mathcal{H}$, defined by: $$ \Phi(\sum_i x_i G_{i+1})=\sum_i \phi(x_i)H_{i+1}.$$
\end{proposition}

\section{The cast} \label{s:cast}

In this section we introduce a host of partially commutative structures associated with a finite graph $\G$. As in the introduction, we will assume that $\G$ has its vertices labelled $v_1,\ldots, v_n$. This is also our preferred generating set of the right-angled Artin group $A_\G$, so that $[v_i,v_j]=1$ in $A_\G$ if there is an edge between the vertices $v_i$ and $v_j$ in $\G$.

\subsection{The monoid $M$ and algebra $U$} \label{s:U}

Let $W(V)$ be the set of positive words in $\{v_1,\ldots,v_n\}$. The empty word is denoted by $\emptyset$ or $1$. We write $|w|$ to denote the length of a word in $W(V)$. We define $\|w\|$, the \emph{multidegree} of a word $w=v_{p_1}^{e_1}\cdots v_{p_k}^{e_k}$ to be the element of $\mathbb{N}^r$ with $i$th coordinate given by $$\sum_{p_j=i}e_j.$$ If $w,w' \in W(V)$, we write $w\leftrightarrow w'$ if there exist $w_1,w_2 \in W(V)$ and vertices  $v_i,v_j$ that are connected by an edge in $\G$ so that: \begin{align*} w&=w_1v_iv_jw_2, \\ w'&=w_1v_jv_iw_2. \end{align*} We then define an equivalence relation on $W(V)$ by saying that $w \sim w'$ if there exist $w_1,\ldots,w_n \in W(V)$ such that $$w=w_1 \leftrightarrow w_2 \leftrightarrow \cdots \leftrightarrow w_n= w'.$$
Let $M=W(V)/\sim$. Let $\overline{w}$ be the equivalence class of $w$ under the equivalence relation $\sim.$ If $w_1\sim w_1'$ and $w_2 \sim w_2'$ then $w_1w_2\sim w_1'w_2'$, therefore multiplication of words in $W(V)$ descends to a multiplication operation on $M$, with an identity element given by the equivalence class of the empty word. Similarly, if $w \sim w'$ then $|w|=|w'|$ and $\|w\|=\|w'\|$, so we may define the length and multidegree of an element $m\in M$ to be the respective length and multidegree of a word in $W(V)$ representing $m$. Length and multidegree are additive with respect to multiplication, so that if $m_1,m_2 \in M$ we have: \begin{align*}|m_1.m_2|&=|m_1|+|m_2| \\ \|m_1.m_2\|&=\|m_1\|+\|m_2\| \end{align*}


This gives the free $\mathbb{Z}$--module on $M$ a graded algebra structure in the following way:

\begin{proposition} Let $U$ be the free $\Z$--module with a basis given by elements of $M.$ Let $U_{i}$ be the submodule of $U$ spanned by the elements of $M$ of length $i$. Then $U=\oplus_{i=0}^{\infty}U_{i}$ and multiplication in $M$ gives $U$ the structure of a graded associative $\mathbb{Z}$--algebra. \end{proposition}

We will distinguish elements of $U$ from $A_\G$ by writing positive words in $\{\mathbf{v_1,\ldots,v_n}\}$ rather than $\{v_1,\ldots,v_n\}$.

\subsection{$U^\infty$, an ideal $X$, and the group of units $U^*$} \label{s:U*}

Let $U^\infty$ be the algebra extending $U$ by allowing infinitely many coefficients
of a sequence of positive elements to be non-zero. Any element of $U^\infty$ can be written uniquely as a power series $a=\sum_{i=0}^{\infty}a_i,$ where $a_i$ is an element of $U_i.$ We say that $a_i$ is the \emph{homogeneous part} of $a$ of degree $i$, and $a_0$ is the \emph{constant term} of $a$. Each $a_i$ is a linear sum of elements of $M_i=\{m\in M : |m|=i\}$, so is of the form $a_i=\sum_{m\in M_i}\lambda_m m$, where $\lambda_m \in\mathbb{Z}$. If $a=\sum_{i=0}^\infty a_i$ and $b=\sum_{i=0}^\infty b_i$ then the homogeneous part of $a.b$ of degree $i$ is $$c_i=\sum_{j=0}^i a_jb_{i-j}.$$

If $a^{(0)},a^{(1)},a^{(2)},\ldots$ is a sequence of elements of $U^\infty$, then the sum $\sum_{j=0}^\infty a^{(j)}$ does not always make sense. However, if the set $$S_i=\{j:a^{(j)}_i \neq 0\}$$ is finite for all $i$ we define $\sum_{j=0}^\infty a^{(j)}$ to be the element of $U^\infty$ with homogeneous part of degree $i$ equal to $$\sum_{j \in S_i} a^{(j)}_i.$$

Let $X$ be the ideal of $U^\infty$ generated by $\mathbf{v_1,\ldots,v_n}$. Alternatively, $X$ is the set of elements of $U^\infty$ with a trivial constant term. In a similar fashion, $X^k$ is the ideal of $U^\infty$ containing all elements $a \in U^\infty$ such that $a_i=0$ for all $i<k.$

Let $U^*$ be the group of units of $U^\infty$. One can show (cf. Proposition \ref{p:inverses}) that $a \in U^*$ if and only if $a= \pm1 +x$ for some $x \in X$. Note that this is much larger than the group of units of $U$: there is an embedding of $A_\G$ into $U^*$ called the \emph{Magnus morphism}, or \emph{Magnus map} (Proposition \ref{prop:mag}).

\section{The Magnus map} \label{s:Magnus map}

To make $U^\infty$ easier to work with, we would like to treat it as a (noncommutative) polynomial algebra. Specifically, we would like to have an idea of `substitution' of elements of $U^\infty$ `into other elements of $U^\infty$'. For instance, given a positive word $w=v_{p_1}\ldots v_{p_k}$ in $W(V)$ and $Q_1,\ldots,Q_n$ in $U^\infty$ we may define $w(Q_1,\ldots,Q_n)=Q_{p_1}\cdots Q_{p_k} \in U^\infty$.  Suppose that $Q_1,\ldots,Q_n$ satisfy \begin{equation}\label{e:substitution}Q_iQ_j=Q_jQ_i \text{ for all $i,j$ such that $v_i$ and $v_j$ span an edge in $\G$.}\end{equation} If $w$ and $w'$ are words such that $w \leftrightarrow w'$ then $$w(Q_1,\ldots,Q_n)=w'(Q_1,\ldots,Q_n).$$ It follows that if $w$ and $w'$ represent the same element of $M$, the above equality also holds. Therefore we may define $m(Q_1,\ldots,Q_n)=w(Q_1,\ldots,Q_n),$ where $w$ is any word in the equivalence class $m$. This definition respects multiplication in $M$, so that for $m_1,m_2 \in M$ we have: \begin{equation}\label{e:mult}m_1(Q_1,\ldots,Q_n)m_2(Q_1,\ldots,Q_n)
=m_1m_2(Q_1,\ldots,Q_n). \end{equation}

We can't quite substitute variables in any element of $U^\infty$ with this level of generality; for example it is not possible to set $x=1$ in $$1 + x  + x^2 + x^3 + \cdots.$$ However, as long as $Q_1,\ldots,Q_n$ have a trivial constant part (in other words they all lie in the ideal $X$) this problem does not occur.

\begin{proposition}\label{p:substitution}
Let $Q_1,\ldots,Q_n$ be elements of $X$ which satisfy condition \eqref{e:substitution}. Then the mapping $$\mathbf{v_i} \mapsto Q_i$$ may be extended to an algebra morphism $\phi:U^\infty \to U^\infty$.
\end{proposition}

\begin{proof} Let $a=\sum_{i=0}^\infty a_i$, with $a_i=\sum_{m \in M_i}\lambda_mm$.  We define: $$\phi(a_i)=\sum_{m \in M_i}\lambda_mm(Q_1,\ldots,Q_n).$$ If $|m|=i$ then as $Q_j \in X$ for all $j$, it follows that $m(Q_1,\ldots,Q_n)$ lies in $X^i$. Therefore the smallest nonzero homogeneous part of $\phi(a_i)$ is of degree at least $i$. Hence the sum $\phi(a)=\sum_{i=1}^\infty \phi(a_i)$ is well defined. It is clear from the definition that $\phi$ is well-behaved under addition and scalar multiplication. Equation \eqref{e:mult} tells us that $\phi$ also behaves well under multiplication, and is an algebra homomorphism. \end{proof}

Such substitutions make our life much easier while working in $U^\infty$; this is exemplified by the following three propositions:  \begin{proposition}\label{p:inverses}If $a$ is of the form $a=1 + \sum_{i=1}^{\infty}a_i,$ then $a \in U^*$ and $$a^{-1}=1-(a_1+a_2+\cdots)+(a_1+a_2+\cdots)^2-\ldots=1 + \sum_{i=1}^{\infty}c_i.$$ Here $c_1=-a_1$ and $c_i=-\sum_{j=0}^{i-1}c_ja_{i-j}=-\sum_{j=1}^{i}a_{j}c_{i-j}$ recursively. \end{proposition}

\begin{proof} One first checks that if $a=1+\mathbf{v_i}$ then the element $a^{-1}=1-\mathbf{v_i} + \mathbf{v_i}^2 - \cdots$ satisfies $a.a^{-1}=a^{-1}.a=1$. We then attain the general formula for an element of the form $a=1+x$ with $x\in X$ by applying the algebra homomorphism given by Proposition \ref{p:substitution} under the mapping $\mathbf{v_i} \mapsto x$ for all $i$. The recursive formula is obtained by equating homogeneous parts in the equation $a^{-1}.a=a.a^{-1}=1.$\end{proof}

\begin{proposition} \label{p:identities} Let $x,y \in X$. Then the following formulas hold: \begin{align} (1+x)(1+y)(1+x)^{-1}&=1+y+(xy-yx)\sum_{i=0}^\infty(-1)^ix^i, \\ \label{e:c1}(1+x)(1+y)(1+x)^{-1}(1+y)^{-1}&=1+(xy-yx)\sum_{i,j=0}^\infty(-1)^{i+j}x^iy^j. \end{align} \end{proposition}

\begin{proof} As in the proof of Proposition \ref{p:inverses}, we first note that these identities hold for $x=\mathbf{v_i}$ and $y=\mathbf{v_j}$ for any $i$ and $j$. For the general case, we wish to apply Proposition \ref{p:substitution}. If $xy=yx$ then we may pick any pair $i, j$ and study the algebra homomorphism induced by the mappings $\mathbf{v_i} \mapsto x$, $\mathbf{v_j} \mapsto y$, and $\mathbf{v_k} \mapsto 0$ when $k\neq i,j$. If $xy \neq yx$ then $M$ is not commutative so $\G$ is not complete: in this case pick vertices $v_i$ and $v_j$ that do not span an edge in $\G$, and use the same map as above.\end{proof}

\begin{proposition}\label{prop:mag}The mapping $v_i \mapsto 1 + \mathbf{v_i}$ induces a homomorphism $\mu:A_\G \to U^*.$ \end{proposition}

\begin{proof} The mapping $v_i \mapsto 1 +\mathbf{v_i}$ induces a homomorphism $\overline{\mu}:F(V) \to  U^*$ from the free group on the set $V$. If $[v_i,v_j]=1$ in $A_\G$ then $\mathbf{v_i}\mathbf{v_j}-\mathbf{v_j}\mathbf{v_i}=0$ in $U^\infty$, therefore by Equation \eqref{e:c1}, relations in the standard presentation of $A_\G$ are sent to the identity in $U^*$, and $\overline{\mu}$ descends to a homomorphism $\mu:A_\G \to U^*.$\end{proof}

The homomorphism $\mu$ is called the \emph{Magnus map}, and is the central object of study in this section. Its first extension to RAAGs was established by Droms~\cite{Droms}, who used it to show that RAAGs are residually torsion-free nilpotent, and we essentially follow his approach here. Our first task is to gain some understanding of the image of a generic element of $A_\G$ under $\mu$.
\begin{definition}
We say that an element $m \in M$ is \emph{square-free} if for all words $w \in W(V)$ representing $m$ there exists no element $v \in V(\G)$ such that $vv$ occurs as a subword of $w$.
\end{definition}

We will now relate square-free elements of $M$ to reduced words representing elements of $A_\G$. (Note that our words representing elements of $A_\G$ are in ${W(V \cup V^{-1})}$ rather than just $W(V)$).

\begin{definition}
Let $g \in A_\G$ and suppose that $w=v_{p_1}^{e_1}\cdots v_{p_k}^{e_k}$ is a word representing $g$ with $e_i \in \mathbb{Z}$. We say that $w$ is \emph{fully reduced} if  $e_i \neq 0$ for all $i$ and for all $j>i$ such that $v_{p_i}=v_{p_j}$ there exists $i<l<j$ such that $v_{p_i}$ and $v_{p_l}$ do not span an edge in $\G$. \end{definition}

We define three moves on the set of words of the form $w=v_{p_1}^{e_1}\cdots v_{p_k}^{e_k}$:
\begin{align*}
&\text{(M1)} \quad \text{Remove $v_{p_i}^{e_i}$ if $e_i=0$.} \\
&\text{(M2)} \quad \text{Replace the subword $v_{p_i}^{e_i}v_{p_{i+1}}^{e_{i+1}}$ with $v_{p_i}^{e_i+e_{i+1}}$ if $p_i=p_{i+1}$.} \\
&\text{(M3)} \quad \text{Replace the subword $v_{p_i}^{e_i}v_{p_{i+1}}^{e_{i+1}}$ with $v_{p_{i+1}}^{e_{i+1}}v_{p_i}^{e_i}$ if $v_{p_i}$ and $v_{p_{i+1}}$ span an edge in $\G$.}
\end{align*}

Given any word $w$ representing $g$ we may find a fully reduced representative of $g$ by applying a sequence of moves of the form (M1), (M2), and (M3). Moves of type (M3) are called \emph{swaps}. If  $w=v_{p_1}^{e_1}v_{p_2}^{e_2}\cdots v_{p_k}^{e_k}$ is fully reduced then $\mathbf{v_{p_1}v_{p_2}\cdots v_{p_k}}$ is square-free. The following key lemma shows that we can find this square-free form in the $k$th homogeneous part of $\mu(g)$. We will use $\mu(g)_i$ to denote the $i$th homogeneous part of $\mu(g)$.

\begin{lemma} \label{injection lemma}
Let $g$ be a nontrivial element of $A_\G$. There exists $k \in \mathbb{N}$ such that $k$ is the largest integer such that there is a square-free element $m \in M_k$ with nonzero coefficient $\lambda_m$ in the decomposition of $\mu(g)_k$. This element is unique. Furthermore, if $v_{p_1}^{e_1}v_{p_2}^{e_2}\cdots v_{p_l}^{e_l}$ is a fully reduced representative of $g$ then $l=k$, $\mathbf{v_{p_1}\cdots v_{p_l}}=m$, and $e_1\cdots e_l=\lambda_m$.
\end{lemma}

\begin{proof}
By an induction argument on $e_i$\,,  we have $$\mu(v_{p_i}^{e_i})=1+e_i\mathbf{v_{p_i}}+\mathbf{v_{p_i}^2}u_i$$ for some $u_i \in U^*.$ Therefore if $v_{p_1}^{e_1}v_{p_2}^{e_2}\cdots v_{p_k}^{e_k}$ is a fully reduced representative of $g$, we have: \begin{align*}\mu(g)&=\mu(v_{p_1}^{e_1})\mu(v_{p_2}^{e_2})\cdots\mu(v_{p_k}^{e_k}) \\
&=(1+e_1\mathbf{v_{p_1}}+\mathbf{v_{p_1}^2}u_1)(1+e_2\mathbf{v_{p_2}}+\mathbf{v_{p_2}^2}u_2)\cdots(1+e_k\mathbf{v_{p_k}}+\mathbf{v_{p_k}^2}u_k).
\end{align*}

In this expansion we see that any positive element occurring with length greater than $k$ must contain $\mathbf{v_{p_i}^2}$ as a subword for some $i$, and the only element of length $k$ without such a subword is $m=\mathbf{v_{p_1}\cdots v_{p_k}},$ with coefficient $\lambda_m=e_1\cdots e_k.$ As $\mu(g)$ is independent of the choice of fully reduced representative of $g$, every fully reduced representative $v_{q_1}^{f_1}\ldots v_{q_l}^{f_l}$ must satisfy $l=k$, with $\mathbf{v_{q_1}\cdots v_{q_l}}=m$ and $f_1\cdots f_l=\lambda_m$. \end{proof}

We have shown that for every nontrivial $g \in A_\G$ there exists $k > 0$ such that $\mu(g)_k$ is nontrivial.

\begin{corollary}
The homomorphism $\mu:A_\G \to U^*$ is injective. 
\end{corollary}

We may now use $\mu$ to study the lower central series of $A_\G$.

\begin{definition} Let $g \in A_\G$.  We define the \emph{derivation} $\delta(g)$ of $g$ to be equal to $\mu(g)_k$, where $k$ is the smallest integer $\geq 1$ such that $\mu(g)_k \neq 0$. If no such $k$ exists, then $g=1$ and we define $\delta(g)=0.$ \end{definition}

The derivation $\delta:A_\G \to U$ satisfies the following properties:

\begin{lemma} \label{l:derivations} Let $g,h \in A_\G$ and suppose that $\delta(g)=\mu(g)_k$ and $\delta(h)=\mu(h)_l$. \begin{enumerate} \item For all integers $N$, $\delta(g^N)=N\mu(g)_k.$ \item If $k<l$ then $\delta(gh)=\delta(hg)=\mu(g)_k.$ \item If $k=l$ and $\mu(g)_k+\mu(h)_l \neq 0$ then $$\delta(gh)=\delta(hg)=\mu(g)_k+\mu(h)_l.$$ \item If $k=l$ and $\mu(g)_k+\mu(h)_l=0$ then either $$gh=1\text{ or } \delta(gh) \in X^{k+1}.$$ \item If $\mu(g)_k\mu(h)_l-\mu(h)_l\mu(g)_k \neq 0$ then $$\delta([g,h])=\mu(g)_k\mu(h)_l-\mu(h)_l\mu(g)_k.$$ \item If $\mu(g)_k\mu(h)_l-\mu(h)_l\mu(g)_k=0$ then either $$[g,h]=0 \text{ or } \delta([g,h])\in X^{k+l+1}.$$ \end{enumerate} \end{lemma}

\begin{proof}  Parts (2) , (3) and (4) follow from standard properties of multiplication in $U^\infty$. Part (1) follows from part (3), an induction argument on $N >0$, and induction on $N<0$. Parts (5) and (6) follow from Equation \eqref{e:c1} in Proposition \ref{p:identities}. \end{proof}

Let $D_k=\{g \in A_\G: \mu(g)_l=0 \text{ if } 0< l < k \}$. Alternatively, $D_k$ is the set of elements $g \in A_\G$ such that either $g=1$ or $\delta(g)\in X^k$.

\begin{proposition}\label{p:dinfo}  For all $k$, the set $D_k$ is a subgroup of $A_\G$ and these subgroups satisfy: \begin{enumerate} \item $\mathcal{D}=\{D_i\}_{i=1}^{\infty}$ is a central filtration of $A_\G$.
\item $D_k/D_{k+1}$ is a finitely generated free abelian group.
\item  $\gamma_k(A_\G) \subset D_k.$ \end{enumerate} \end{proposition}

\begin{proof} Parts (2)--(4) of Lemma \ref{l:derivations} imply that $D_k$ is a subgroup of $A_\G$. By definition, $D_1=A_\G$ and $D_{k+1} \leq D_k$ for all $k$.  Also, if $g \in D_k$ and $h \in D_l$, then $[g,h] \in D_{k+l}$ by parts (5) and (6) of Lemma \ref{l:derivations}. Therefore $\mathcal{D}=\{D_i\}$ satisfies the requirements (F1), (F2) and (F3) given in Section \ref{s:Lie} and is a central filtration of $A_\G$. For part (2), we define the map $\phi:D_k \to U_k$ by defining $\phi(g)=\mu(g)_k.$ Equivalently:
$$\phi(g)=\begin{cases} \delta(g) & \text{if $\delta(g)=\mu(g)_k$} \\ 0 &\text{otherwise, when $\delta(g) \in X^{k+1}$.} \end{cases}$$ Parts (2)--(4) of Lemma \ref{l:derivations} imply that $\phi$ is a homomorphism to $U_k$, with kernel $D_{k+1}$. Therefore the quotient group $D_k/D_{k+1}$ is isomorphic to a subgroup of $U_k$. As $U_k$ is a finitely generated free abelian group, so is $D_k/D_{k+1}$. Part (3) is satisfied for all central filtrations of $A_\G$ by Proposition~\ref{p:centr}. \end{proof}

As $\mathcal{D}$ is a central filtration of $A_\G$, we have $\g_i(A_\G) \subset D_i$ for all $i$, and as the Magnus map is injective, $\cap_{i=1}^{\infty} D_i =\{ 1\}$. Hence we may apply Proposition~\ref{p:cf} to the central filtration $\mathcal{D}$ to obtain:

\begin{theorem}
The intersection $\cap_{i=1}^\infty\g_i(A_\G)=\{1\}$ and $A_\G$ is residually torsion-free nilpotent.\end{theorem} 

 We finish this section with a proof of a normal form theorem for elements of $A_\G$.  This is well-known; Green's thesis \cite{Green} contains a combinatorial proof involving case-by-case analysis. Green's work also extends more generally to graph products of groups. We give a proof for RAAGs using the Magnus map. The first step is an immediate consequence of Lemma \ref{injection lemma}:
 
 \begin{proposition}\label{p:lengths}
Let $g \in A_\G$. Let $w=v_{p_1}^{e_1}\cdots v_{p_k}^{e_k}$ and $w'=v_{q_1}^{f_1}\cdots v_{q_l}^{f_l}$ be two fully reduced representatives of $g$. Then $k=l$.
\end{proposition}

In fact, we can prove something much more detailed:

\begin{theorem} \label{t:nform}
Let $g \in A_\G$. Let $w=v_{p_1}^{e_1}\cdots v_{p_k}^{e_k}$ and $w'=v_{q_1}^{f_1}\cdots v_{q_k}^{f_k}$ be two fully reduced representatives of $g$. Then we may obtain $w$ from $w'$ by a sequence of swaps (moves of the form $v_{p_i}^{e_i}v_{p_{i+1}}^{e_{i+1}} \mapsto v_{p_{i+1}}^{e_{i+1}}v_{p_i}^{e_i}$ when $[v_{p_i},v_{p_{i+1}}]=1$). \end{theorem}

\begin{proof}
We proceed by induction on $k$. We first look at the element $v_{p_1}^{-e_1}g \in A_\G.$ Note that $v_{p_2}^{e_2}\cdots v_{p_k}^{e_k}$ and $v_{p_1}^{-e_1}v_{q_1}^{f_1}\cdots v_{q_k}^{f_k}$ are two representatives of $v_{p_1}^{-e_1}g$, and the former representative is fully reduced. By Proposition \ref{p:lengths} the latter cannot be fully reduced, so there exists $l$ such that ${q_l}={p_1}$ and $[v_{p_1},v_{q_i}]=1$ for $i \leq l$. If $f_l \neq e_1$, then $$v_{q_1}^{f_1}\cdots v_{q_l}^{f_l-e_1} \cdots v_{q_k}^{f_k}$$  is a fully reduced representative of $v_{p_1}^{-e_1}g$, however this also contradicts Proposition \ref{p:lengths}. Therefore $e_1=f_l$, and after applying a sequence of swaps to $w'$ we may assume that $v_{p_1}=v_{q_1}$ and $e_1=f_1$. By induction, $v_{p_2}^{e_2}\cdots v_{p_k}^{e_k}$ may be obtained from $v_{q_2}^{f_2}\cdots v_{q_k}^{f_k}$ by a sequence of swaps, therefore $w$ may be obtained from $w'$ by a sequence of swaps. \end{proof}

Given $g \in A_\G$, let $init(g)$ (respectively $term(g)$) be the set of vertices of $\G$ that can occur as the initial (respectively terminal) letter of a fully reduced word representing $g$. We say that $g$ is \emph{positive} if $g=1$ or $g$ can be written as a product $v_1^{e_1}\cdots v_k^{e_k}$ with $e_i > 0$ for all $i$. As any two fully reduced representatives may be obtained from each other by a sequence of swaps, we have the following immediate corollaries:

\begin{corollary}
For any $g \in A_\G \smallsetminus \{1\}$, the sets $init(g)$ and $term(g)$ form cliques in $\G$\,: any pair of vertices in $init(g)$ or $term(g)$ commute. \end{corollary}
 
\begin{corollary}
The monoid $M$ is isomorphic to the set of positive elements of $A_\G$ under multiplication. \end{corollary}



\section{Lyndon elements of $M$}\label{s:Lyndon}

Let $\mathcal{L}(U)$ be the Lie algebra we obtain by endowing $U$ with the bracket operation $[a,b]=ab-ba$. We will now study the Lie subalgebra of $\mathcal{L}(U)$ generated by the set $\{\mathbf{v_1,\ldots,v_n}\}$. We call this subalgebra $L_\G$. The approach is as follows: we first introduce a subset of $M$ called the set of \emph{Lyndon elements}, $LE(M)$. We describe a method for supplying each Lyndon element with a bracketing. If $L$ is a Lie algebra and $X=\{x_1,\ldots,x_n\}\subset L$ then this bracketing induces a homomorphism (as $\mathbb{Z}$--modules) $\phi_X:\mathbb{Z}[LE(M)]\to L$. In the case that $X=\{\mathbf{v_1},\ldots,\mathbf{v_n}\} \subset L_\G$ we call this induced homomorphism $\ell$, and show that $\ell$ is bijective. Thus we obtain a basis of $L_\G$ in terms of bracketed Lyndon elements. In general, if the elements of $X$ satisfy $$[x_i,x_j]=0 \text{ if } [v_i,v_j]=1$$ then we will show that $\phi_X \ell^{-1}:L_\G \to L$ is an algebra homomorphism taking $\mathbf{v_i}$ to $x_i$. This property will then be used in the next section to show that $L_\G$ and the lower central series algebra of $A_\G$ 
are isomorphic. 

We deviate here from the approach in Magnus, and instead follow the paper of Lalonde \cite{L1}. The analogous free group version is contained in Chapter 5 of \cite{Lothaire}, and we must start in this world. We first define a lexicographic order on the set of positive words $W(V)$:

\begin{definition}
The \emph{lexicographic ordering} on $W(V)$ is the unique total order $<$ on $W(V)$ that satisfies the following:
\begin{enumerate}
\item For any nonempty word $w$, we have $\emptyset < w$.
\item If $w_1$ and $w_2$ are distinct nonempty words and $x,y \in W(V)$ such that $w_1=v_ix$ and $w_2=v_jy$, then $w_1 <w_2$ if either
\begin{enumerate}
\item $i<j$ or:
\item $i=j$ and $x <y$.
\end{enumerate}
\end{enumerate}
\end{definition}

In particular, $\emptyset < v_1 < v_2 < \ldots < v_n$. We state two basic properties of this order:

\begin{lemma} \label{l:o1}
Let $x,y,z \in W(V)$. \begin{itemize}\item if $y < z$ then $xy < xz$ \item 
if $|x| \geq |y|$ and $x<y$ then $xz<yz$ \end{itemize} 
\end{lemma}

The above lemma remains valid if we replace all occurrences of strong inequalities with weak inequalities. The natural projection $\pi:W(V) \to M$, when coupled with the ordering of $W(V)$, gives us a way of choosing a representative in $W(V)$ for each element of $M$: 

\begin{definition} Let $m \in M$. Then we define $std(m)\in W(V)$, the \emph{standard representative} of $m$ to be the largest element of $\pi^{-1}\{m\}$ with respect to the lexicographic order. \end{definition}

\begin{example} 
Let $\G$ be the small example graph given in Figure 1. If $m \in M$ is the element represented by the word $v_1v_2v_3$ then then $\pi^{-1}(m)=\{v_1v_2v_3, v_1v_3v_2\}$ and $std(m)=v_1v_3v_2$.
\end{example}

\begin{figure} [ht]
\includegraphics{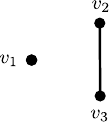}
\centering
\caption{A small example graph $\G$.}
\label{fig:minigraph}
\end{figure}

We then define a total order on $M$ as follows: if $a,b\in M$  we say $$a < b \text{ if and only if } std(a)<std(b).$$ Lemma \ref{l:o1} then implies the following:

\begin{lemma} \label{l:ordering} Let $a,b,c \in M$ 
\begin{itemize} \item $std(ab) \geq std(a)std(b) \geq std(a)$ \item If $b < c$ then $std(a)std(b)<std(a)std(c)$ \item If $|a| \geq |b|$ and $a < b$, then $std(a)std(c) < std(b)std(c)$ \end{itemize} \end{lemma}

\subsection{Lyndon words}

We now describe the notion of \emph{Lyndon words}. These were first introduced by Chen, Fox, and Lyndon in \cite{CFL}. In this paper, the authors show that in the free group case, the groups $D_k$ introduced in the last section are equal to the terms of the lower central series of $F_n$, and they give an algorithm to determine a presentation of a consecutive quotient $\gamma_k/\gamma_{k+1}$ of the lower central series for any finitely presented group. This algorithm is quite complicated, however we shall use the notion of \emph{Lyndon elements} in $M$, introduced by Lalonde in \cite{L1}, to give a simple algorithm to describe $\gamma_k/\gamma_{k+1}$ in an arbitrary right-angled Artin group. Chen, Fox, and Lyndon also relate coefficients of elements in $\mu(g)$ to \emph{Fox derivatives}. Unfortunately these do not appear to have a natural analogue in the partially commutative setting.

We say that $w_1$ and $w_2$ are \emph{conjugate} in $W(V)$ if there exist $x,y \in W(V)$ such that $w_1=xy$ and $w_2=yx$. Alternatively, $w_1$ and $w_2$ are conjugate if they are conjugates in the free group $F_n$ in the usual sense, where $W(V)$ is viewed as a subset of $F_n$. The \emph{conjugacy class} of $w$ in $W(V)$ is the set of all elements conjugate to $w$ in $W(V)$. A word $w$ is \emph{primitive} if there does not exist $x,y \in W(V) \setminus \{\emptyset\}$ such that $w=xy=yx$. Equivalently, each nontrivial cyclic permutation of $w$ is distinct from $w$.

\begin{definition}
$w \in W(V)$ is a \emph{Lyndon word} if it is nontrivial, primitive and minimal with respect to the lexicographic ordering in its conjugacy class. \end{definition}

\begin{example}If $V=\{v_1,v_2,v_3,v_4\}$ then $v_i$ is a Lyndon word for all $i$, and $v_1v_2v_1v_3$ and $v_1v_1v_2$ are Lyndon words. The word $v_1v_1$ is not a Lyndon word as it is not primitive, and $v_1v_3v_1v_2$ is not a Lyndon word as it is not minimal in its conjugacy class (the word $v_1v_2v_1v_3$ is).
\end{example}

There is an assortment of equivalent definitions of Lyndon words.

\begin{theorem}[\cite{CFL}, Theorem 1.4] \label{t:lwdefinitions}
Let $w \in W(V)$. The following are equivalent: \begin{enumerate} \item $w$ is a Lyndon word. \item For all $x,y \in W(V)\smallsetminus \{\emptyset\}$ such that $w=xy$, one has $w<y$. \item Either $w=v_i$ for some $i$ or there exist Lyndon words $x$ and $y$ with $x <y$ such that $w=xy$. \end{enumerate} \end{theorem}

The third of these characterisations is particularly appealing, as it allows one to build up a list of Lyndon words recursively.

\begin{example}
If $V=\{v_1,v_2,v_3\}$, then the Lyndon words of length less than or equal to 3 are: \begin{align*} &v_1,\, v_2,\, v_3, \\
& v_1v_2,\, v_1v_3,\, v_2v_3, \\ &v_1v_1v_2,\, v_1v_1v_3,\, v_1v_2v_3,\, v_2v_2v_3,\, v_1v_2v_2,\, v_1v_3v_3,\, v_2v_3v_3. \end{align*}

Note that the decomposition of a Lyndon word of length $>1$ as a product of two smaller Lyndon words assured to us by part (3) of Theorem \ref{t:lwdefinitions} is not always unique. In this example $v_1v_2v_3$ may be decomposed as $v_1.v_2v_3$ and $v_1v_2.v_3$.
\end{example}

\subsection{Lyndon elements}

Lyndon elements are the natural generalisations of Lyndon words to the partially commutative setting. Defining conjugation here is more tricky. We first say that two elements $m_1, m_2$ of $M$ are \emph{transposed} if there exist $x,y \in M$ such that $m_1=xy$ and $m_2=yx$. Unfortunately transposition is not an equivalence relation; if $\G$ is the graph shown in Figure~\ref{fig:minigraph}, then $$ v_2v_1v_3 \leftrightarrow_{trans.} v_1v_3v_2 = v_1v_2v_3 \leftrightarrow_{trans.} v_3v_1v_2, $$ however $v_3v_1v_2$ cannot be obtained from $v_2v_1v_3$ by a single transposition. We therefore say two elements of $M$ are \emph{conjugate} if one can be obtained from the other by a sequence of transpositions. Equivalently, two elements are conjugate in $M$ if and only if they are conjugate in $A_\G$ in the group theoretic sense (when $M$ is viewed as a subset of $A_\G$). The set of all elements in $M$ conjugate to $m$ is its \emph{conjugacy class}. We say that $m$ is primitive if there do not exist nontrivial $x$ and 
$y$ in $M$ such that $m=xy=yx$.

\begin{definition}
$m \in M$ is a \emph{Lyndon element} if it is nontrivial, primitive, and minimal with respect to the ordering of $M$ in its conjugacy class. \end{definition}

Given $g \in A_\G$, we remind the reader that $init(g)$ is the set of vertices that can appear as the initial letter in reduced words representing $g$. 

\begin{proposition}[\cite{KL}, Corollary 3.2]
If $m$ is a Lyndon element, then $init(m)$ is a single vertex.
\end{proposition}

Given $m \in M$, we say that $v_i \in \zeta(m)$ if either $v_i \in supp(m)$ or there exists $v_j \in supp(m)$ such that $[v_i,v_j] \neq 1$. Equivalently $v_i \in \zeta(m)$ if and only if either $v_i \in supp(m)$ or $v_im \neq mv_i$. In a similar fashion to Lyndon words, there is a selection of equivalent definitions of Lyndon elements.

\begin{theorem}[\cite{KL}, Propositions 3.5, 3.6, and 3.7] \label{t:ledefinitions}
Let $m \in M$. The following are equivalent.

\begin{enumerate} 
\item $m$ is a Lyndon element.
\item For all $x,y \in M \smallsetminus \{1\}$ such that $m=xy$, one has $m<y$.
\item Either $|m|=1$ or there exist Lyndon elements $x,y$ such that $x<y$, $init(y) \in \zeta(x)$ and $m=xy$.
\item $std(m)$ is a Lyndon word.
\end{enumerate}
\end{theorem}

Once again, the third part of the classification gives a simple recursive process for writing down Lyndon elements. 

\begin{example}
If $\G$ is the small example graph of Figure~\ref{fig:minigraph}, then the Lyndon elements of length $\leq 3$ are: \begin{align*} &v_1,\, v_2,\, v_3 \\ &v_1v_2,\, v_1v_3 \\ &v_1v_1v_2,\, v_1v_1v_3,\, v_1v_2v_2,\, v_1v_2v_3,\, v_1v_3v_3 \end{align*}

The words given here are a subset of the set of Lyndon words on $\{v_1,v_2,v_3\}$. So for example, the Lyndon word $v_2v_3$ does not represent a Lyndon element as $v_3 \not \in \zeta(v_2)=\{v_1,v_2\}$. As with Lyndon words, the decomposition of a Lyndon element of length $>1$ as a product of two Lyndon elements is not necessarily unique. In this example $v_1v_2v_3$ has two possible decompositions as $v_1v_2.v_3$ and $v_1v_3.v_2$. 
\end{example}

\subsection{The standard factorisation of a Lyndon element}

We now give each Lyndon element a unique `bracketing'. If $m$ is a Lyndon element of length greater than 1, there may exist many pairs of Lyndon elements $x$ and $y$ such that $m=xy$. If $y$ is minimal as we run through all such pairs of Lyndon elements in $M$, we say that $S(m)=(x,y)$ is the \emph{standard factorisation} of $m$. The standard factorisation behaves well with respect to standard representatives:

\begin{theorem}[\cite{L2}, Proposition 2.1.10]
Let $a \in M$ be a Lyndon element. If $S(a)=(x,y)$ is the standard factorisation of $a$, then $std(a)=std(x)std(y)$.
\end{theorem}

Note that if $x,y$ are two Lyndon elements of $M$ with $x<y$ then $std(x)$, $std(y)$ and $std(x)std(y)$ are Lyndon words. As each Lyndon word is strictly less than its nontrivial conjugates, \begin{equation}\label{e:lyndon} std(x)std(y) < std(y)std(x).\end{equation} We shall use this trick repeatedly in the work that follows. There is one final combinatorial fact we need before we can move on:

\begin{theorem}[\cite{L2}, Proposition 2.3.9]\label{t:standard}
Suppose that $a$ and $b$ are Lyndon elements with $a<b$ and $init(b)\in \zeta(a)$ (in particular, $ab$ is a Lyndon element). Then $S(ab)=(a,b)$ if and only if $|a|=1$ or $S(a)=(x,y)$ and $y \geq b.$
\end{theorem}

\begin{example}\label{e:bracket}We now have a recursive way of giving a bracketing to any Lyndon element. Given $m \in M$, take its standard factorisation $S(m)=(a,b)$, and define the bracketing on $m$ to be equal to $[[a],[b]]$, where $[\_]$ denotes the bracketing on $a$ and $b$ respectively. In our small example graph, the only interesting case is $std(v_1v_2v_3)=v_1v_3v_2=std(v_1v_3)std(v_2)$. We then obtain the following bracketing on Lyndon elements of length 3: $$[v_1,[v_1,v_2]], \, [v_1,[v_1,v_3]], \, [[v_1,v_2],v_2], \, [[v_1,v_3],v_2], \, [[v_1,v_3],v_3].$$
\end{example}
\subsection{A basis theorem for the algebra $L_\G$}

Let $LE(M)$ be the set of Lyndon elements of $M$. Let $\mathbb{Z}[LE(M)]$ be the free $\mathbb{Z}$--module with basis $LE(M).$

\begin{definition} \label{d:phi} Let $L$ be a Lie algebra, and let $X=\{x_1,x_2,\ldots,x_n\}$ be a subset of $L$. Let $\phi_X:\mathbb{Z}[LE(M)] \to L$ be the $\mathbb{Z}$--module homomorphism defined recursively as follows:
\begin{align*} \phi_X(\mathbf{v_i})&=x_i &&\text{for all $i$}\\ \phi_X(a)&=[\phi_X(x),\phi_X(y)] && \text{if $|a|>1$ and $S(a)=(x,y)$.} \end{align*}
\end{definition} 

\begin{example} Let $L_\G$ be the Lie subalgebra of $\mathcal{L}(U)$ generated by the set $V=\{\mathbf{v_1},\ldots,\mathbf{v_n}\}$. We attain a $\mathbb{Z}$--module homomorphism $\phi_V:\mathbb{Z}[LE(M)] \to L_\G$. We write $\phi_V=\ell$. The map $\ell$ can be thought of as the bracketing procedure for Lyndon elements described above. \end{example}

The following technical lemma gives us a way of understanding the bracket operation in $\mathcal{L}(U)$.

\begin{lemma}\label{l:mult}
Suppose that $f=\sum_{b \in I} \alpha_b b$ and $g=\sum_{c \in J} \beta_c c$ are two homogeneous elements in $U^{\infty}$, so that $|b|=|b'|$ for all $b,b' \in I$ and $|c|=|c'|$ for all $c \in J$. Let $x$ be the minimal element in $I$ with $\alpha_x$ nonzero and $y$ be the minimal element in $J$ with $\beta_y$ nonzero. Suppose that $x$ and $y$ are Lyndon elements, $x<y$ and $init(y) \in \zeta(x)$, so that $xy$ is a Lyndon element. Then \begin{itemize} \item $[f,g]$ is a homogeneous element of $U^{\infty}$ of degree $|xy|$; \item $xy$ is the minimal element of $M$ with nonzero coefficient in $[f,g]$; \item The coefficient of $xy$ in $[f,g]$ is $\alpha_x\beta_y$. \end{itemize}
Furthermore, if $f$ and $g$ are homogeneous with respect to multidegree, so that $\|b\|=\|b'\|$ for all $b,b' \in I$ and $\|c\|=\|c'\|$ for all $c,c' \in J$, then $[f,g]$ is homogeneous with respect to the multidegree $\|xy\|$.
\end{lemma}

\begin{proof}
We have: \begin{equation}\label{e:prod} [f,g]=\sum_{b \in I}\sum_{c \in J}\alpha_b\beta_c (bc -cb), \end{equation} where we may assume that $b \geq x$ and $c \geq y$, and $|bc|=|cb|=|xy|$. If either $b>x$ or $c>y$ then by Lemma~\ref{l:ordering}: \begin{align*} std(bc) &\geq std(b)std(c) \\ &> std(x)std(y) \\ &=std(a).\end{align*} By the identities in Lemma \ref{l:ordering} and the identity \eqref{e:lyndon} we also have: \begin{align*} std(cb) &\geq std(c)std(b) \\ &\geq std(y)std(x)\\ &> std(x)std(y) \\ &=std(a). \end{align*} Hence $cb > xy$ for all $b \in I, c \in J$ and $bc \geq xy$ with equality if and only if $b=x$ and $c=y$, so the coefficient of $a$ in the above sum is $\alpha_x\beta_y$. The final remark about homogeneity with respect to multidegree follows as if $f$ and $g$ are homogeneous with respect to multidegree then $\|bc\|=\|cb\|=\|xy\|$ for all $b\in I$ and $c\in J$. \end{proof}

\begin{proposition}\label{p:x}
For each $a \in LE(M)$, there exists a subset $I \subset M$ and a set of nonzero integers $\{\alpha_b\}_{b \in I}$ indexed by $I$ such that $$\ell(a) =\sum_{b \in I} \alpha_b b.$$ Furthermore, $a \in I$ with $\alpha_a=1$, and for all $b \in I$ we have $\|b\|=\|a\|$ and $b \geq a$. \end{proposition}

\begin{proof}
We proceed by induction on $|a|$. If $|a|=1$ then $\ell(a)=a$ and we are done. Suppose that $|a|> 1$. Let $S(a)=(x,y)$ be the standard decomposition of $a$. By our inductive hypothesis we may write $$\ell(x)=\sum_{b \in I}\alpha_bb \text{ and } \ell(y)=\sum_{c\in J}\beta_cc$$ with $b \geq x$, $c \geq y$ and $\|b\|=\|x\|$, $\|c\|=\|y\|$ for all $b \in I$ and $c \in J$. Furthermore we may assume  $\alpha_x=\beta_y=1$. As $\ell(a)=[\ell(x),\ell(y)]$ the result follows from Lemma~\ref{l:mult}.  \end{proof}

A consequence of the above theorem is that the image of $LE(M)$ under $\ell$ forms a linearly independent set.

\begin{corollary}\label{c:inj}
The map $\ell:\mathbb{Z}[LE(M)] \to L_\G$ is injective. \end{corollary} 

We now go back to the more general situation.

\begin{lemma}\label{l:zero}
Let $L$ be a Lie algebra, and suppose that $X=\{x_1,\ldots,x_n\}$ is a subset of $L$ that satisfies $$[x_i,x_j]=0 \text{ when } [v_i,v_j]=1.$$ Suppose that $a$ is a Lyndon element of $M$, and $v_i \in V$ such that $[\mathbf{v_i},a]=0$ in $U$. If $\phi_X$ is defined as in Definition~\ref{d:phi}, then $$[\phi_X(a),\phi_X(\mathbf{v_i})]=0.$$ \end{lemma}

\begin{proof}
We induct on the length of $a$. If $a=\mathbf{v_j}$ for some $j$ then $[v_i,v_j]=1$. Therefore $[\phi_X(a),\phi_X(v_i)]=[x_j,x_i]=0$. If $|a|>1$ then $S(a)=(x,y)$ for some $x,y \in LE(M)$ such that $[x,\mathbf{v_j}]=[y,\mathbf{v_j}]=0$. Therefore by induction $[\phi_X(\mathbf{v_i}),\phi_X(x)]=[\phi_X(y),\phi_X(\mathbf{v_i})]=0$, and by the Jacobi identity in $L$: \begin{align*} [\phi_X(a),\phi_X(\mathbf{v_i})]&=[[\phi_X(x),\phi_X(y)],\phi_X(\mathbf{v_i})] \\ &=-[[\phi_X(\mathbf{v_i}),\phi_X(x)],\phi_X(y)]-[[\phi_X(y),\phi_X(\mathbf{v_i})],\phi_X(x)] \\&=-[0,\phi_X(y)]-[0,\phi_X(x)] \\ &=0. \qedhere \end{align*}\end{proof}

What follows is the main technical theorem of this section, which will allow us to extend the $\mathbb{Z}$--module homomorphism $\phi_X$ to something that behaves well with respect to brackets also.

\begin{proposition} \label{p:key}
Let $L$ be a Lie algebra, and suppose that $X=\{x_1,\ldots,x_n\}$ is a subset of $L$ that satisfies $$[x_i,x_j]=0 \text{ if } [v_i,v_j]=1.$$ Let $\phi_X$ be the homomorphism defined in Definition \ref{d:phi}. Let $a,b \in LE(M)$ be such that $a<b$. Then there exists a subset $I_{a,b} \subset LE(M)$ and a set of integers $\{\alpha_c\}_{c \in I_{a,b}}$ indexed by $I_{a,b}$ such that $$[\phi_X(a),\phi_X(b)]=\sum_{c \in I_{a,b}} \alpha_c \phi_X(c).$$ Furthermore, each $c \in I_{a,b}$ satisfies the following: \begin{enumerate}[(B1)] \item $c <b$, \item $std(c)\geq std(a)std(b)$, \item $\|c\|=\|ab\|$, \end{enumerate} and the sets $I_{a,b}$ and $\{\alpha_c\}_{c \in I_{a,b}}$ are independent of $L$ and $X$.\end{proposition}

\begin{proof}
The first step is to define an order $\prec$ on the set of pairs $(a,b) \in LE(M)\times LE(M)$ satisfying $a<b$. We say $(a,b) \prec (a',b')$ if \begin{itemize} \item $|ab| < |a'b'|$, or \item $|ab|=|a'b'|$ and $std(a)std(b) > std(a')std(b')$, or \item $std(a)std(b)=std(a')std(b')$ and $b<b'$. \end{itemize} Note that the second criterion is possibly the reverse of what one might expect. We shall prove Proposition \ref{p:key} by using induction on the order given by $\prec$. We drop the subscript of $\phi_X$ for the remainder of this proof. The base case is when $(a,b)=(v_{n-1},v_n)$ and is trivial. The inductive step splits into two cases.

\textbf{Case 1.} $init(b) \in \zeta(a)$.

If $|a|=1$, then Theorem \ref{t:standard} tells us $S(ab)=(a,b)$, and $[\phi(a),\phi(b)]=\phi(ab)$ by definition. Also, $ab<b$ by part 2 of Theorem \ref{t:ledefinitions}, and $std(ab) \geq std(a)std(b)$.

If $|a|>1$, let $S(a)=(x,y)$. This now splits into two subcases. 

\textbf{Subcase 1.} $y \geq b$. By Theorem \ref{t:standard}, we have $S(ab)=(a,b)$, and we are in exactly the same situation as case 1.

\textbf{Subcase 2.} $y<b$. We use the Jacobi identity in L: \begin{align*} [\phi(a),\phi(b)]&=[[\phi(x),\phi(y)],\phi(b)]\\ &=-[[\phi(b),\phi(x)],\phi(y)]-[[\phi(y),\phi(b)],\phi(x)] \\&=[[\phi(x),\phi(b)],\phi(y)]+[\phi(x),[\phi(y),\phi(b)]] \end{align*}

We look at the two parts of this sum separately.
$$\textbf{The }[[\phi(x),\phi(b)],\phi(y)] \textbf{ part:}$$ Note that $x<a<b$, and $|xb|<|ab|$, so we have $(x,b)\prec (a,b)$. Therefore by induction there exists a decomposition: $$[\phi(x),\phi(b)]=\sum_{c\in I_{x,b}} \alpha_c\phi(c)$$ with each $c$ satisfying (B1)--(B3) with respect to $(x,b)$. Then for each $c$, if $y<c$ then \begin{align*} std(y)std(c) &\geq std(y)std(x)std(b) &&\text{by (B2)}\\ &> std(x)std(y)std(b) &&\text{by (8)} \\ &=std(a)std(b), \end{align*} so that $(y,c) \prec (a,b)$. If $y=c$ then $[\phi(y),\phi(c)]=0$. If $c<y$ then as $std(c)\geq std(x)std(b)$ and $std(y)<std(b)$ we have: \begin{align*} std(c)std(y) &\geq std(x)std(b)std(y) \\ &> std(x)std(y)std(b) \\ &=std(a)std(b), \end{align*} so that $(c,y) \prec (a,b).$ In any case, by induction there exists a decomposition: $$[\phi(c),\phi(y)]=\sum_{d \in I_{c,y}} \beta_d\phi(d)$$ with each $d$ satisfying (B1)--(B3) with respect to either $(y,c)$ or $(c,y)$. As the $c$ here satisfies (B1)--(B3) with respect to $(x,b)$ one can check that each 
$d$ also satisfies (B1)--(B3) with respect to $(a,b)$ and we have the required decomposition:
$$[[\phi(x),\phi(b)],\phi(y)]=\sum_{c \in I_{x,b}}\sum_{d \in I_{c,y}} \alpha_c\beta_d \phi(d).$$ 

 $$\textbf{The } [\phi(x),[\phi(y),\phi(b)]] \textbf{ part:}$$ Since $y<b$ and $|yb|<|ab|$ there exists a decomposition $[\phi(y),\phi(b)]=\sum_{c\in I_{y,b}} \alpha_cc$ with each $c$ satisfying (B1)--(B3) with respect to $(y,b)$. Also for each $c$ we have \begin{align*} std(c) &\geq std(y)std(b) \\ &\geq std(y) \\ &>std(x) ,\end{align*} so that $x<c$ and \begin{align*} std(x)std(c) &\geq std(x)std(y)std(b) \\ &=std(a)std(b). \end{align*} Hence $(x,c) \prec (a,b)$, and by induction we have the decomposition $$[\phi(x),\phi(c)]=\sum_{d\in I_{x,c}} \beta_d \phi(d)$$ with each $d$ satisfying (B1)--(B3) with respect to $(x,c)$. As $c <b$ and $std(x)std(c) \geq std(a)std(b)$ each $d$ also satisfies (B1)--(B3) with respect to $(a,b)$. This gives our required decomposition $$[\phi(x),[\phi(y),\phi(b)]]=\sum_{c \in I_{y,b}}\sum_{d \in I_{x,c}} \alpha_c\beta_d \phi(d)$$
Adding the above two parts gives the required decomposition of $[\phi(a),\phi(b)],$ and finishes the inductive step in this first case.

\textbf{Case 2.} $init(b) \not \in \zeta(a)$.

If $|b|=1$ then $[\phi(a),\phi(b)]=0$ by Lemma \ref{l:zero}, and we are done. If $|b|>1$, then we write $S(b)=(x,y)$. By the Jacobi identity in $L$:
\begin{align*} [\phi(a),\phi(b)]&=[\phi(a),[\phi(x),\phi(y)]] \\ &=-[\phi(x),[\phi(a),\phi(y)]] - [\phi(y),[\phi(x),\phi(a)]] \\ &=[[\phi(a),\phi(y)],\phi(x)]-[[\phi(a),\phi(x)],\phi(y)]. \end{align*}

Again we look at the two separate parts in this sum. First, $[[\phi(a),\phi(y)],\phi(x)]$. As $(a,y) \prec (a,b)$ by induction there exists a decomposition $$[\phi(a),\phi(y)] = \sum_{c \in I_{a,y}} \alpha_c \phi(c),$$ with each $c$ satisfying (B1)--(B3) with respect to $(a,y)$. We would like to show that $c <x$ and $(c,x) \prec (a,b)$. Note that the smallest letter (with respect to the ordering $v_1 < v_2 < \cdots <v_n$) of any Lyndon word must be its initial letter, otherwise there would be a conjugate of that word that is smaller with respect to the ordering of $M$. Let $inf(g)$ denote the smallest letter in $supp(g)$ for any $g \in M$. As $\|c\|=\|ay\|$, we have: $$init(c)=inf(c)=inf(ay) \leq inf(a)=init(a) < init(b)=init(x).$$ The strict inequality holds in the above as $a <b$ and $init(a)\neq init(b)$ because $init(b) \not \in \zeta(a)$. Hence $c<x$, and \begin{align*} std(c)std(x) &\geq std(a)std(y)std(x) \\ &> std(a)std(x)std(y) \\ &=std(a)std(b) \end{align*} Therefore $(c,x) \prec (a,b)$, and there is a 
decomposition $$[\phi(c),\phi(x)] = \sum_{d \in I_{c,x}} \beta_d \phi(d),$$ with each $d$ satisfying the required (B1)--(B3) with respect to $(c,x)$. Once again it is not hard to check that $d$ also satisfies (B1)--(B3) with respect to $(a,b)$.  For $[[\phi(a),\phi(x)],\phi(y)]$ the same methods apply as before and we will spare the reader any further details. 

This completes the induction proof. The only part we have not covered is the fact that the sets $I_{a,b}$ and $\{\alpha_c\}_{c \in I_{a,b}}$ are independent of $X$ and $L$, however this is clear as we did not need use our choice of $L$ or $X$ at any point in the proof. \end{proof}

Proposition \ref{p:key} implies that the image of $\ell$ in $L_\G$ is closed under the bracket operation, so is a subalgebra of $L_\G$. As $L_\G$ is the smallest subalgebra of $\mathcal{L}(U)$ containing $\{\mathbf{v_1,\ldots,v_n}\}$, and this set is in the image of $\ell$, this means that $\ell$ is surjective. We have shown in Corollary \ref{c:inj} that $\ell$ is also injective.

\begin{corollary}
The map $\ell:\mathbb{Z}[LE(M)] \to L_\G$ is bijective.
\end{corollary}

For our toils, we can now show that $L_\G$ satisfies the following universal property:

\begin{theorem}\label{t:univprop}
Let $\G$ be a graph with vertices $v_1,\ldots,v_n$. Let $L$ be a Lie algebra, and suppose that $X=\{x_1,\ldots,x_n\}$ is a subset of $L$ that satisfies: $$[x_i,x_j]=0 \text{if $v_i$ and $v_j$ are connected by an edge in $\G$.}$$ Then there is a unique algebra homomorphism $\psi_X:L_\G \to L$ such that $$\psi_X(\mathbf{v_i})=x_i \text{ for $1 \leq i \leq r$.}$$ 
\end{theorem}

\begin{proof}
As $L_\G$ is generated by $V$, if such a map exists then it is unique. Let $\psi_X=\phi_X \ell^{-1}$. As $\psi_X$ is a $\mathbb{Z}$--module morphism, we only need to check the bracket operation on the basis $\ell(LE(M))$ of $L_\G$. Let $a,b \in LE(M)$ and without loss of generality suppose that $a <b$. By Proposition \ref{p:key} there exists $I \subset LE(M)$ and a set of integers $\{\alpha_c\}_{c \in I}$ such that \begin{align*} [\ell(a),\ell(b)]&=\sum_{c\in I}\alpha_c \ell(c) \\ \text{and } [\phi_X(a),\phi_X(b)]&=\sum_{c \in I} \alpha_c \phi_X(c). \end{align*}
Therefore \begin{align*} \psi_X([\ell(a),\ell(b)])&=\psi_X(\sum_{c\in I}\alpha_c \ell(c)) \\ &=\sum_{c\in I}\alpha_c\psi_X\ell(c) \\ &=\sum_{c \in I} \alpha_c\phi_X(c) \\ &=[\phi_X(a),\phi_X(b)] \\ &=[\psi_X(\ell(a)),\psi_X(\ell(b))]. \qedhere \end{align*}
\end{proof}

\section{An isomorphism between $L_\G$ and the LCS algebra of $A_\G$}\label{s:end}

The algebra $L_\G$ inherits a grading from $\mathcal{L}(U)$ by letting $L_{\G,i}=L_\G \cap \mathcal{L}(U)_i $. We note that $$L_{\G,i}=\langle \ell(a) : a \in LE(M), |a|=i \rangle.$$

Previously we defined $\mathcal{C}$ and $\mathcal{D}$ to be the linear filtrations of $A_\G$ given by the lower central series, and the central series $\{D_i\}$ given in section \ref{s:Magnus map} respectively.

\begin{lemma}Let $X=\{ v_i\g_1(A_\G) : 1 \leq i \leq n \} \subset L_\mathcal{C}$. The algebra homomorphism $\psi_X:L_\G \to L_\mathcal{C}$ given by Theorem \ref{t:univprop} respects the gradings of $L_\G$ and $L_\mathcal{C}$. \end{lemma}

\begin{proof} We show that $\psi_X(L_{\G,k}) \subset L_{\mathcal{C},k}$ by induction on $k$. As $\psi_X(\mathbf{v_i})=v_i\g_1(A_\G)$, and $L_{\G,1}$ is spanned by $\{\mathbf{v_1},\ldots,\mathbf{v_n}\}$, the case $k=1$ holds. For the inductive step, pick $a \in LE(M)$ such that $|a|=k$. Let $S(a)=(b,c)$ be the standard decomposition of $a$, with $|b|=i$, $|c|=j$, and $i+j=k$. Then by induction $\psi_X(\ell(b)) \in L_ {\mathcal{C},i}$ and $\psi_X(c) \in L_{\mathcal{C},j}$, hence \begin{equation*} \psi_X(\ell(a))=[\psi_X(\ell(b)),\psi_X(\ell(c))]\in L_{\mathcal{C},i+j}=L_{\mathcal{C},k}. \qedhere \end{equation*} \end{proof}

By Proposition \ref{p:dinfo} we know that $\g_k(A_\G) \subset D_k$ for all $k$. Hence by Proposition~\ref{p:func} the identity map $A_\G\to A_\G$ induces a graded algebra homomorphism ${\alpha:L_\mathcal{C} \to L_\mathcal{D}}.$ 

\begin{lemma} The mapping $ g D_{k+1} \mapsto \mu(g)_k$  induces a graded algebra homomorphism $\beta:L_\mathcal{D} \to \mathcal{L}(U)$. \end{lemma}

\begin{proof} The group $D_{k+1}$ is the kernel of the homomorphism $D_k \to U_k$ given by $g \mapsto \mu(g)_k$. Therefore the induced map $\beta:L_\mathcal{D} \to \mathcal{L}(U)$ is well-defined. As $\mu(g)_k \in \mathcal{L}(U)_k$, this map also respects gradings. The fact that $\beta$ is a homomorphism is implied by parts (1), (5) and (6) of Lemma \ref{l:derivations}. \end{proof}

We now have a chain of graded algebra homomorphisms $$L_\G \xrightarrow{\psi_X} L_\mathcal{C} \xrightarrow{\alpha} L_\mathcal{D} \xrightarrow{\beta} \mathcal{L}(U),$$ which allows us to prove our main theorem.

\begin{theorem} \label{t:main}
$L_\G$, $L_\mathcal{C}$, and $L_\mathcal{D}$ are isomorphic as graded Lie algebras. Furthermore, the central filtrations $\mathcal{C}$ and $\mathcal{D}$ are equal, so that $\gamma_k(A_\G)=D_k$ for all $k \geq 1$.
\end{theorem}

\begin{proof}
We start by calculating the image of $\{\mathbf{v_1},\ldots,\mathbf{v_n}\}$ under $\beta\alpha\psi_X$. We have \begin{align*} \beta\alpha\psi_X(\mathbf{v_i})&=\beta\alpha(v_i\g_1(A_\G)) \\ &=\beta(v_iD_1) \\ &=\mu(v_i)_1 \\&=\mathbf{v_i}. \end{align*}

Therefore the image of $\beta\alpha\psi_X$ is $L_\G$, and as $\beta\alpha\psi_X$ takes the generators to themselves, it is the identity map on $L_\G$. In particular, $\psi_X$ must be injective. By Proposition \ref{p:gen}, the algebra $L_\mathcal{C}$ is generated by the set $X$, hence $\psi_X$ is also surjective, and is an isomorphism. We now know that $L_\mathcal{C}$ and $L_\G$ are isomorphic as graded Lie algebras. Then $\beta\alpha$ maps $L_\mathcal{C}$ isomorphically onto $L_\G$, so the map $\alpha$ is also injective. Looking at each graded piece, each homomorphism $$\g_k(A_\G)/\g_{k+1}(A_\G) \xrightarrow{\alpha_k} D_k/D_{k+1}$$ is injective. We shall use this to show that $\gamma_k(A_\G)=D_k$ by induction on $k$, and this will complete the proof of the main theorem. Note that $D_1=\gamma_1(A_\G)$ by definition. Suppose that $\gamma_k(A_\G)=D_k$. Then $\alpha_k$ is also surjective, so is an isomorphism. If $g \in D_k=\gamma_k(A_\Gamma)$, then \begin{align*} g \in D_{k+1} &\iff gD_{k+1}=1 && \text{in $D_k/D_{k+1}$} \\ &\iff \alpha_k^{-1}(gD_{k+1})=1 &&\text{in $\g_k(A_\G)/\gamma_{k+1}(A_\G)$} \\  &\iff g\gamma_{k+1}(A_\G)=1 &&\text{in $\g_k(A_\G)/\gamma_{k+1}(A_\G)$} \\ &\iff g \in \gamma_{k+1}(A_\G). \end{align*} Hence $\gamma_{k+1}(A_\G)=D_{k+1}$. 
\end{proof}

We conclude with an important consequence of Theorem \ref{t:main} and Proposition \ref{p:dinfo}:

\begin{theorem} \label{t:tfn1}
If $k \in \mathbb{N}$, then $\g_k(A_\G)/\g_{k+1}(A_\G)$ is free-abelian, and $A_\G/\g_k(A_\G)$ is torsion-free nilpotent.
\end{theorem}

\begin{example}
Let $\G$ be the small example graph given in Figure 1. We have already worked out the bracketing of Lyndon elements of length 3 in example \ref{e:bracket}. The isomorphism given in Theorem~\ref{t:main} tells us that $\g_3(A_\G)/\g_4(A_\G)$ is freely generated by $[v_1,[v_1,v_2]]\g_4(A_\G)$, $[v_1,[v_1,v_3]]\g_4(A_\G)$, $[[v_1,v_2],v_2]\g_4(A_\G)$, $[[v_1,v_3],v_2]\g_4(A_\G)$, and $[[v_1,v_3],v_3]\g_4(A_\G)$.
\end{example}

\bibliography{bib}{}
\bibliographystyle{plain}

\end{document}